\documentclass[12pt, reqno]{amsart} 
\usepackage{amsmath, amssymb, amsfonts, amstext, amsthm, amscd}
\usepackage[mathscr]{euscript}
\usepackage{enumerate}
\usepackage{tikz,color,soul}
\usepackage[english]{babel}
\usepackage{chngcntr}
\usepackage{stmaryrd}
\usepackage{a4wide}

\usepackage[pagebackref]{hyperref}

\usepackage[left=2.5cm,top=2.5cm,bottom=3cm,right=2.5cm]{geometry}

\usetikzlibrary{arrows.meta, positioning,arrows}
\makeatletter
\newtheorem*{rep@theorem}{\rep@title}
\newcommand{\newreptheorem}[2]{%
	\newenvironment{rep#1}[1]{%
		\def\rep@title{#2 \ref{##1}} 
		\begin{rep@theorem}}%
		{\end{rep@theorem}}}
\makeatother

\newreptheorem{conjecture}{Conjecture}

\newtheorem{theorem}{Theorem}[section]

\newtheorem{prop}[theorem]{Proposition}
\newtheorem{lem}[theorem]{Lemma}

\theoremstyle{definition}
\newtheorem{definition}[theorem]{Definition}

\newtheorem{remark}[theorem]{Remark}

\newcommand{\Li}{\mathcal{L}}
\newcommand{\Lia}{\Li(a,b,m_1,\dots,m_r)}
\newcommand{\Fe}{\mathbb{F}_e}
\newcommand{\Fer}{\mathbb{F}_{e,r}}

\theoremstyle{plain}
\numberwithin{equation}{section}
\newtheorem*{theorem*}{Theorem}
\newtheorem{proposition}[theorem]{Proposition}
\newtheorem{lemma}[theorem]{Lemma}

\newtheorem{conjecture}[theorem]{Conjecture}
\newtheorem{algorithm}[theorem]{Algorithm}
\theoremstyle{definition}
\newtheorem{question}[theorem]{Question}
\newtheorem{example}[theorem]{Example}

\numberwithin{equation}{section}
\title{Seshadri constants on  blow-ups of Hirzebruch surfaces}

\author[K. Hanumanthu]{Krishna Hanumanthu}
\address{Chennai Mathematical Institute, H1 SIPCOT IT Park, Siruseri, Kelambakkam 603103, India}
\email{krishna@cmi.ac.in}

\author[C. J.  Jacob]{Cyril J. Jacob}
\address{Chennai Mathematical Institute, H1 SIPCOT IT Park, Siruseri, Kelambakkam 603103, India}
\email{cyril@cmi.ac.in}

\author[Suhas B. N.]{Suhas B. N.}
\address{Chennai Mathematical Institute, H1 SIPCOT IT Park, Siruseri, Kelambakkam 603103, India}
\email{suhasbn@cmi.ac.in}

\author[A. K. Singh]{Amit Kumar Singh}
\address{Department of Mathematics, SRM University AP, Amaravati 522 240, Andhra Pradesh, India}
\email{amitkumar.si@srmap.edu.in}

\subjclass[2020]{14C20,14E05,14J26}

\keywords{Seshadri constant, Hirzebruch surfaces, Linear system of curves}
\date{\today}
\begin{document}
	
	\begin{abstract}
		Let $e,r \ge 0$ be integers and let $\mathbb{F}_e : = \mathbb{P}(\mathcal{O}_{\mathbb{P}^1} \oplus \mathcal{O}_{\mathbb{P}^1}(-e))$ denote the Hirzebruch surface with invariant $e$. 
		We compute the Seshadri constants of an ample line bundle at an arbitrary point of the $r$-point blow-up of 
		$\mathbb{F}_e$  when $r \leq e-1$ and at a very general point when $r=e$ or $r=e+1$. We also discuss several conjectures on linear systems of curves on the blow-up of $\mathbb{F}_e$ at $r$ very general points.
	\end{abstract}
	
	\maketitle
	
	\section{Introduction}\label{Introduction}

	Let $\mathbb{K}$ denote an algebraically closed field of characteristic $0$. All the varieties considered throughout this article are defined over $\mathbb{K}$. 
	
	Seshadri constants measure the local positivity of an ample line bundle on a projective variety. They were first introduced by Demailly as a way to study the Fujita conjecture \cite{Dem92}. Before giving the formal definition of Seshadri constants,  we recall the following criterion for ampleness given by Seshadri.
	\begin{theorem} \cite[Theorem I.7.1]{Hartshorne1970}
		Let $X$ be a projective variety and $L$ be a line bundle on $X$. Then $L$ is ample if and only if there exists a positive number $\varepsilon$ such that for all $x \in X$ and all irreducible and reduced curves $C$ passing through $x$, one has $L \cdot C \geq \varepsilon.\text{mult}_xC$.
	\end{theorem}
	
	The question of finding the optimal number $\varepsilon$ satisfying the above condition naturally leads to the following definition.
	
	\begin{definition} [Seshadri constant at a point]
		Let $X$ be an irreducible projective variety and $L$ be a nef line bundle on $X$. Then for  $x \in X$, the real number
		\begin{equation*}
			\varepsilon(X,L;x) := \inf \frac{L \cdot C}{\text{mult}_x C}
		\end{equation*}
		is called the \textit{Seshadri constant} of $L$ at $x$, where the infimum is taken over all irreducible and reduced curves $C$ such that $x \in C$.
	\end{definition}
	
	\begin{definition}
		{Let $X$ be an irreducible projective variety and $L$ be a nef line bundle on $X$. If $x\in X$ and $C$ is a curve passing through $x$, then the ratio $\frac{L\cdot C}{\text{mult}_xC}$ is called the \textit{Seshadri ratio of $L$}  with respect to $C$ and $x$.  If  $\varepsilon(X,L,x)=\frac{L\cdot C}{\text{mult}_xC}$ for a curve $C$ (i.e., infimum is achieved by $C$), then $C$ is called a \textit{Seshadri curve of $L$ at $x$}. } 
	\end{definition}
	
	There is an equivalent definition of Seshadri constant of $L$ at $x \in X$ in terms of blow-ups. Let $x \in X$ and let $\pi : \widetilde{X} = \text{Bl}_x(X) \rightarrow X$ be the blow-up of $X$ at $x$ with exceptional divisor $E \subset \widetilde{X}$. Then it is not hard to see the following equality: 
	\begin{equation}\label{2nd definition of Seshadri constant}
		\varepsilon(X,L;x) = \text{sup}\{s \geq 0 \;|\; \pi^{*}L - sE \,\text{ is nef}\}.
	\end{equation}
	
	In general, Seshadri constants are hard to compute and their precise values are known only in a few cases. One of the main challenges involved in their computation is that they depend on the individual curve and not just on the linear or algebraic equivalence class of the curve. However, even if one cannot compute their explicit value, one hopes to provide some (optimal) bounds on them. Let $n$ denote the dimension of $X$ and let $L$ be any ample line bundle on $X$.  
	The following inequalities are well-known and hold for every point $x \in X$:
	
	$$0<\varepsilon(X,L;x) \leq \sqrt[n]{L^n}.$$ In specific cases, one can ask if better bounds can be obtained.

	There is a rich literature on the study of Seshadri constants on various smooth projective surfaces (for example, see \cite{Har1986}, \cite{Ein-Lazarsfeld1993}, \cite{Bauer2009}, \cite{Sano2014}). In particular, on the $r$-point blow-up of $\mathbb{P}^2$, the explicit computation of Seshadri constants for ample line bundles was carried out in \cite{Sano2014} for $r \leq 8$. 
	
	We now turn our attention to \textit{Hirzebruch surfaces}. Let $e \geq 0$ be an integer. 
	The projective bundle $$\mathbb{F}_e : = \mathbb{P}(\mathcal{O}_{\mathbb{P}^1} \oplus \mathcal{O}_{\mathbb{P}^1}(-e))$$ over $\mathbb{P}^1$
	associated to the rank 2 vector bundle $\mathcal{O}_{\mathbb{P}^1}\oplus \mathcal{O}_{\mathbb{P}^1}(-e)$ is called the \textit{Hirzebruch surface with invariant} $e$. Let $C_e$ denote a section of the natural map $\mathbb{F}_e \to \mathbb{P}^1$ such that $C_e$ is the divisor associated to the line bundle $\mathcal{O}_{\mathbb{F}_e}(1)$. 
	%	 curve on $\mathbb{F}_e$ with self intersection $-e$.  
	Let $f$ denote the class of a fiber of  map 
	$\mathbb{F}_e \to \mathbb{P}^1$. Then the Picard group 
	$\text{Pic}(\mathbb{F}_e)$
	of 
	$\mathbb{F}_e$ is generated by $C_e$ and $f$ and the intersection product of divisors on  $\Fe$ is determined by the following:
	\[C_e^2=-e, f^2=0,  \text{ and } C_e\cdot f = 1.\]

	Precise values of Seshadri constants are known for any ample line bundle on 
	$\mathbb{F}_e$; see \cite[Theorem 5.1]{Garcia2006} and \cite[Theorem 3.27]{Syzdek2005}.  Let $L = aC_e + bf$ be an ample line bundle on $\mathbb{F}_e$ where $a,b$ are positive integers. 
	
	When $e =0$, the Seshadri constants are given by $$\varepsilon(\mathbb{F}_e,L;x) = \min (a,b) \text { for all }x \in \mathbb{F}_e.$$ 
	
	When $e \geq 1$, we have 
	$$\varepsilon(\Fe,L;x)=
	\begin{cases}
		\begin{split}
			\min (a, \, &b-ae),  &\text{ if } x\in C_e,\\
			&a, 	&\text{ if }x\notin C_e.
		\end{split}
	\end{cases} $$

	It is well-known that $\mathbb{P}^2$ and $\mathbb{F}_e$ with $e \neq 1$ are the minimal rational surfaces \cite[Theorem V.10]{Beau1983}. Any rational surface can be obtained as a blow-up of a minimal rational surface at finitely many points. Seshadri constants of ample line bundles on blow-ups of $\mathbb{P}^2$ are extensively studied; for example, see \cite{Sano2014, Dumnicki2016, Hanumanthu2018, Farnik2020}. 
	
	In view of this, it is natural to ask if one can obtain analogous results on blow-ups of $\mathbb{F}_e$. 
	
	Let $p_1,\dots, p_r \in \mathbb{F}_e$  be distinct points and $f_i$ the fiber in $\mathbb{F}_e$ containing $p_i$ for each $i$.  Let 
	$$\pi : \mathbb{F}_{e,r} \rightarrow \mathbb{F}_e$$ be the blow-up of $\Fe$ at $p_1,\dots, p_r$. 
	%It is known that the Picard group of $\mathbb{F}_e$ is a free abelian group generated by the class of $C_e$ and the class of a fiber $f$. 
	%Intersection of divisors in $\Fe$ is determined by $C_e^2=-e, f^2=0$ and $C_e\cdot f = 1$. 
	Then  $$\text{Pic}(\mathbb{F}_{e,r}) = \text{Pic}(\mathbb{F}_e) \oplus \mathbb{Z}. E_1  \oplus \dots \oplus \mathbb{Z}. E_r,$$ where $E_i$ are the exceptional divisors in $\mathbb{F}_{e,r}$. Moreover, in both $\mathbb{F}_e$ and $\mathbb{F}_{e,r}$, the numerical equivalence coincides with the linear equivalence, so that the Picard groups of these surfaces are the same as the corresponding \textit{N\'eron-Severi} groups, which we denote by $\text{NS}(\mathbb{F}_{e})$ and $\text{NS}(\mathbb{F}_{e,r})$, respectively. Let $K_{\mathbb{F}_e}$ and $K_{\mathbb{F}_{e,r}}$ denote the canonical line bundles of $\mathbb{F}_e$ and $\mathbb{F}_{e,r}$, respectively. We have $$ K_{\mathbb{F}_e} \sim -2 C_e - (e+2) f $$ and $$K_{\mathbb{F}_{e,r}} \sim -2 H_e - (e+2) F_e + \sum\limits_{i=1}^r E_i,$$ where $H_e = \pi^*(C_e)$ and $F_e = \pi^*(f).$

	By abuse of notation, we will use $H_e$ and $F_e$ to denote divisor classes as well as specific curves linearly equivalent to these divisor classes. 	
	
	For a positive integer $m$, a \textit{$(-m)$-curve} in $\mathbb{F}_{e,r}$ is a reduced and irreducible curve $C$ such that $C^2 = -m$ and $K_{\mathbb{F}_{e,r}}\cdot C = m-2$.\\
	
	The first aim of this article is to compute Seshadri constants of ample line bundles on $\Fer$ for certain values of $r$. One of our main results is the following. 
	% see Theorem \ref{actual computation of Seshadri constants} and Theorem \ref{for x in general position}.
	\begin{theorem*}[See Theorem \ref{actual computation of Seshadri constants} and Theorem \ref{for x in general position}]
		Let $p_1,\dots, p_r \in \mathbb{F}_e$ be very general points, where $r \leq e+1$. Let $L = aH_e+bF_e-\sum\limits_{i=1}^r m_iE_i$ be an ample line bundle on $\mathbb{F}_{e,r}$ and $x \in \mathbb{F}_{e,r}$ a very general point. %, where $a,b,m_i \in \mathbb{Z}$. 
		Then 
		\[
		\varepsilon(\Fer,L;x)=
		\begin{cases}
			\begin{split}
				&a, &\text{ if } &r\leq e-1,\\
				\min (a, \, &b-\sum m_i),  &\text{ if }&r=e \text{ or } e+1,
			\end{split}
		\end{cases} 
		\] 
		where the sum runs over the largest $e$ numbers among $\{m_1,\dots, m_r\}$.
	\end{theorem*}
	Further, when $r \leq e-1$, for a suitable choice of $r$ points, we compute $\varepsilon(\Fer,L;x)$ for all $x \in \Fer$ (see Theorem \ref{actual computation of Seshadri constants}).\\
	
	There are several famous conjectures in literature on  linear systems of curves in $\mathbb{P}^2$; Nagata Conjecture, SHGH conjecture \cite{Ser1961, Har1986, Gim1987, Hir1989}, Weak SHGH conjecture, to name a few. These conjectures also have connections  to Seshadri constants on blow-ups of $\mathbb{P}^2$. 
	
	It is natural to study similar questions on $\Fe$. The version of SHGH given in \cite{Ser1961} has been reformulated for Hirzebruch surface by Laface \cite{Laf2002} (see Conjecture \ref{conj1}). Dumnicki showed that this conjecture holds for imposed base points of equal multiplicity bounded by $8$ (\cite[Theorem 6]{Duminicki2010}). We study these questions in Section \ref{LinearSystem}.   Following \cite{Gim1987}, we propose the following conjecture.
	\begin{repconjecture}{conj2}
		Let  $\widetilde{C}_e$ be the strict transform of $C_e$ on $\mathbb{F}_{e,r}$.
		If a general curve of a non-empty linear system $\Li$ on 
		$\mathbb{F}_{e,r}$ is reduced and $\widetilde{C}_e$ is not a fixed component of $\Li$, then $\Li$ is non-special.
	\end{repconjecture}
	
	We also propose another conjecture on $\mathbb{F}_{e,r}$ which resembles the Weak SHGH conjecture for 
	$\mathbb{P}^2$ (see Conjecture \ref{conj3}). We prove this conjecture for $r \leq e+2$. We also study 
	relations between these conjectures. At the end of the article, we ask a  question which relates these conjectures to 
	the irrationality of Seshadri constants.
	\section{Preliminaries}\label{Preliminaries}
	
	We now state a few important results that will be used later.
	
	\begin{theorem}\cite[Corollary 1.2]{Ein-Lazarsfeld1993} \label{EinLars}
		Let $X$ be a surface and $x\in X$ be a very general point. For a reduced and irreducible curve $C$ on $X$ passing through $x$, we have
		$$C^2\geq m(m-1),$$
		where $m=\text{mult}_xC$.
	\end{theorem}
	\begin{proof}
		Since $C$ passes through a very general point with multiplicity $m$ there is a non-trivial family $(C_t,x_t)_{t\in T}$ of irreducible and reduced curves $C_t$ and points $x_t \in C_t$ parametrized by a variety $T$ such that $C_t$ is numerically equivalent to $C$ and $\text{mult}_{x_t}C_t \ge m$ for every $t \in T$ and $C_{t_0} = C$, $x_{t_0} = x$ for some $t_0 \in T$. For more details about this family, see \cite[Proposition 2.5]{Kol1995} and 
		\cite[Proof of Theorem 2, Page 209]{Nak2003}.  
		
		Since the family $(C_t,x_t)_{t\in T}$ is non-trivial, the Kodaira-Spencer map for a general element of $T$ is non-zero. Then we can apply \cite[Corollary 1.2]{Ein-Lazarsfeld1993} to obtain the desired inequality. 
	\end{proof} 
	Next lemma generalises the above theorem to the multi-point case:
	\begin{lemma}\label{Xumulti}
		Let $X$ be a smooth projective surface and $p_1,p_2,\dots, p_r$ be very general points on $X$. Let $C$ be an irreducible reduced passing through at least one $p_i$ and let $m_i=\text{mult}_{p_i}C$. Then
		$$C^2\geq \left(\sum_{i=1}^{r}m_i^2\right)-m_j$$
		for any $m_j>0$.
		\begin{proof}
			Assume without loss of generality that $j=r$ and $m_r>0$.
			Let $\pi:X_{r-1}\to X$ be the blow up of $X$ at $p_1,p_2,\dots p_{r-1}$. Then the strict transform of $C$ on $X_{r-1}$ is 
			$$\widetilde{C}=\pi^\ast(C)-m_1E_1-m_2E_2-\dots-m_{r-1}E_{r-1}.$$
			So $\widetilde{C}$ passes through $x:=\pi^{-1}(\{p_r\})$ and $\text{mult}_x\widetilde{C}=m_r$. As $x$ is a very general point of $X_{r-1}$,
			using Theorem \ref{EinLars}, we have
			$$\widetilde{C}^2\ge \text{mult}_x\widetilde{C}(\text{mult}_x\widetilde{C}-1)=m_r(m_r-1).$$
			Note that $\displaystyle \widetilde{C}^2=C^2-\sum_{i=1}^{r-1}m_i^2$.
			Hence
			$$C^2\ge \sum_{i=1}^{r}m_i^2-m_r.$$		\end{proof}
	\end{lemma}

	\begin{prop}
		\label{cartier divisor becoming nef} 
		Let $S$ be a smooth rational surface with Picard rank $\rho(S) \ge 3$. Assume that $S$ is anti-canonical; that is, $|-K_S| \neq \emptyset$, where $K_S$ is the canonical divisor class on $S$. Then we have the following. \\
		\begin{itemize}
			\item[(i)] $\overline{NE}(S) = \mathbb{R}_{\geq 0}[-K_S] + \sum\limits_{\substack{C \subset S~ {\rm an} \\ {\rm irred. curve} \\{\rm with }~ C^2 < 0}}\mathbb{R}_{\geq 0}[C]$. \\
			\item[(ii)] Let $C \subset S$ be an irreducible reduced curve such that $C^2 < 0$. Then $C$ is either a $(-1)$-curve, a $(-2)$-curve or a fixed component of $|-K_S|$.\\
			\item[(iii)]  Let $D$ be a Cartier divisor on $S$. Then $D$ is nef if and only if $D\cdot C \geq 0$ for all $C$ such that $C$ is either a $(-1)$-curve, a $(-2)$-curve, a fixed component of $|-K_S|$ or in $|-K_S|$.
		\end{itemize}
	\end{prop}
	\begin{proof}	
		\begin{itemize}
			\item[(i)] This follows from \cite[Lemma 4.1]{Lahyane-Harbourne2005}.
			\item[(ii)] The proof of this can be found in \cite[Proposition 4.1 (ii)]{Sano2014}.
			\item[(iii)] This follows from (i) and (ii).
		\end{itemize}
	\end{proof}
	
	%	(i) This follows from \cite[Lemma 4.1]{Lahyane-Harbourne2005}.\\
	%	(ii) The proof of this can be found in \cite[Proposition 4.1 (ii)]{Sano2014}. \\
	%	(iii) This follows from (i) and (ii).
	As $\mathbb{F}_{e,r}$ is a smooth rational surface, Proposition \ref{cartier divisor becoming nef} applies to it if it is anti-canonical. This is true in certain cases:   
	\begin{prop} \cite[Proposition 3]{Medina-Lahanye2020} \label{Condition for a smooth rational surface to be anti-canonical}
		The surface $\mathbb{F}_{e,r}$ is anti-canonical if $r \leq e+5$.   
		\begin{remark}
			In \cite{Medina-Lahanye2020}, it is assumed that the $r$ points that are getting blown-up are in general position. However, we wish to emphasize that the same proof goes through for $r$ distinct points, not necessarily general.  
		\end{remark}
		% \begin{itemize}
			%     \item[(i)] $r \leq e+5$,\\
			%     \item[(ii)] $1 \leq e \leq 2$ and $r = e+6$, or \\
			%     \item[(iii)] $e = 1$ and $r =8$.
			% \end{itemize}
	\end{prop}

	Suppose that $\pi_x : \widetilde{X} = \text{Bl}_x(X) \rightarrow X$ is the blow-up of $X$ at a point $x$ and $E_x$ is the corresponding exceptional divisor in $\widetilde{X}$. Then by  \eqref{2nd definition of Seshadri constant}, we know that for a nef line bundle $L$ in $X$, the Seshadri constant $\varepsilon(X, L; x)$ is obtained by taking the supremum of all non-negative real numbers $s$ such that $\pi_x^*(L) - sE_x$ is nef in $\widetilde{X}$. However, by (iii) of Proposition \ref{cartier divisor becoming nef}, to say that $\pi_x^*(L) - sE_x$ is nef in $\widetilde{X}$, it is enough to check it on $(-1)$-curves, $(-2)$-curves, curves which are fixed components of $|-K_{\widetilde{X}}|$ and curves in the linear system $|-K_{\widetilde{X}}|$, provided the surface is anti-canonical. Let $\mathcal{S}'$ denote the set consisting of all $(-1)$-curves, $(-2)$-curves, curves which are fixed components of $|-K_{\widetilde{X}}|$ and curves in the linear system $|-K_{\widetilde{X}}|$. Then we have the following result.
	\begin{prop}\label{compute Seshadri constant}
		Let $X$ be a smooth rational surface with Picard rank $\rho(X) \ge 2$. Let $x \in X$ and $L$ be a nef line bundle on $X$. If $\widetilde{X}$ is anti-canonical then the Seshadri constant of $L$ at $x$ is given by 
		\begin{equation}\label{1epsilon1}
			\varepsilon(X, L; x) = \inf \Bigg\{\frac{L\cdot C}{{\rm mult}_x C}~ \Bigg \vert ~ \widetilde{C} \in \mathcal{S}'\Bigg\},
		\end{equation}
		where $\mathcal{S}'$ is defined as above and $C$ is a curve in $X$ with $\widetilde{C}$ as its strict transform in $\widetilde{X}$.
	\end{prop}
	\begin{proof} As $X$ has Picard rank at least two, $\widetilde{X}$ has Picard rank at least three. So
		by part (iii) of Proposition \ref{cartier divisor becoming nef}, we have 
		\begin{equation}\label{new expression for epsilon}
			\varepsilon(X, L; x) = \text{sup} ~ \{s \in \mathbb{R}~\vert~(\pi_x^*(L) - sE_x)\cdot C^{'}\geq 0, ~~ \forall~ C^{'} \in \mathcal{S}' \}.
		\end{equation}
		Let us denote this number by $\varepsilon$, for convenience. Let $\varepsilon_{1}$ denote the number on the right hand side of  \eqref{1epsilon1}. Note that $\varepsilon_1 \geq 0$ as $L$ is nef. We claim that $\varepsilon = \varepsilon_{1}$. 
		
		Let $s \in \mathbb{R}$ be such that $(\pi_x^*(L) - sE_x)\cdot C^{'}\geq 0,~ \forall~ C^{'} \in \mathcal{S}'$. In particular, this implies that for every curve $C \subset X$ such that $\widetilde{C} \in \mathcal{S}'$, $(\pi_x^*(L) - sE_x)\cdot \widetilde{C}\geq 0$. This means
		\[
		\begin{split}
			(\pi_x^*(L) - sE_x)\cdot (\pi_x^*(C)-(\text{mult}_xC)E_x)  & \geq 0  \\
			\implies L\cdot C - s~ \text{mult}_xC & \geq 0 \\
			\implies \frac{L\cdot C}{\text{mult}_xC} & \geq s \\
			\implies \frac{L\cdot C}{\text{mult}_xC} & \geq \varepsilon \\
			\implies \varepsilon_1 & \geq \varepsilon.
		\end{split}
		\]
		Conversely, for $C \subset X$ such that $\widetilde{C} \in \mathcal{S}'$, consider $(\pi_x^*(L) - \varepsilon_1 E_x)\cdot \widetilde{C}$. We have
		\begin{eqnarray*}
			(\pi_x^*(L) - \varepsilon_1 E_x)\cdot \widetilde{C} & = & (\pi_x^*(L) - \varepsilon_1 E_x)\cdot (\pi_x^*(C)-(\text{mult}_xC)E_x)  \\
			& = &L\cdot C - \varepsilon_1 ~ \text{mult}_xC~\geq 0,  
		\end{eqnarray*}
		because $\varepsilon_1$ is, by definition, the infimum of all the ratios of the form $\frac{L\cdot C}{\text{mult}_x C}$, where $\widetilde{C} \in \mathcal{S}'$. Also, if $E_x \in \mathcal{S}'$, then $(\pi_x^*(L) - \varepsilon_1 E_x)\cdot E_x = \varepsilon_1 \geq 0$. This means for every $C^{'} \in \mathcal{S}'$, $(\pi_x^*(L) - \varepsilon_1 E_x)\cdot  C^{'} \geq 0$. This in turn implies $\varepsilon \geq \varepsilon_1$, as $\varepsilon$ is the supremum of all such values. This completes the proof.
	\end{proof}

	\section{Seshadri constants}\label{Computation of Seshadri constants}
	Let	 $\pi_x : \widetilde{\mathbb{F}_{e,r}} = \text{Bl}_x(\mathbb{F}_{e,r}) \rightarrow \mathbb{F}_{e,r}$ be the blow-up of $\mathbb{F}_{e,r}$ at a point $x$ and let $E_x$ be the corresponding exceptional divisor in $\widetilde{\mathbb{F}_{e,r}}$.
	Let $\mathcal{S}'$ denote a set of curves in $\widetilde{\mathbb{F}_{e,r}}$
	consisting of the following:  $(-1)$-curves, $(-2)$-curves, curves which are fixed components of
	$|-K_{\widetilde{\mathbb{F}_{e,r}}}|$ and curves in the linear system $|-K_{\widetilde{\mathbb{F}_{e,r}}}|$. 
	Let $\mathcal{S} \subset \mathcal{S}'$ denote the set consisting of all $(-1)$-curves, $(-2)$-curves and curves which are the fixed components of $|-K_{\widetilde{\mathbb{F}_{e,r}}}|$. 
	Then we have the following result.
	\begin{prop}\label{our way to compute Seshadri constant}
		Let $e \ge 0$, $r \leq e+4$, $x \in \mathbb{F}_{e,r}$.  Let $L = aH_e + b F_e - \sum\limits_{i=1}^r m_i E_i$ be a nef line bundle on $\mathbb{F}_{e,r}$. Then the Seshadri constant of $L$ at $x$ is given by 
		\begin{equation}\label{epsilon1}
			\varepsilon(\mathbb{F}_{e,r}, L; x) = \inf \Bigg\{\frac{L\cdot C}{{\rm mult}_x C}~ \Bigg \vert ~ \widetilde{C} \in \mathcal{S}\Bigg\},
		\end{equation}
		where $\mathcal{S}$ is as defined just before the proposition and $C$ is a curve in $\mathbb{F}_{e,r}$ with $\widetilde{C}$ as its strict transform on $\widetilde{\mathbb{F}_{e,r}}$.
	\end{prop}
	\begin{proof}
		Since $r \leq e + 4$, Proposition \ref{Condition for a smooth rational surface to be anti-canonical} implies that $\widetilde{\mathbb{F}_{e,r}}$ is anti-canonical.  Let $f_{\pi(x)}$ denote the fiber in $\Fe$ containing $\pi(x)$, where  $\pi: \Fer \to \Fe$ is the natural blow-up morphism.  Let $f_x$ denote the strict transform of $f_{\pi(x)}$
		in $\mathbb{F}_{e,r}$. Then  $L \cdot f_x \le  a$. So $\varepsilon(\mathbb{F}_{e,r}, L; x) \leq a$.
		% So by Proposition \ref{compute Seshadri constant}, this proposition holds. 
		
		Note that by Proposition \ref{compute Seshadri constant}, 	
		$\varepsilon(\mathbb{F}_{e,r}, L; x)$ is the infimum of Seshadri ratios with respect to curves whose strict transforms are in $\mathcal{S}'$.  
		If $C \subset \mathbb{F}_{e,r}$ is a curve whose strict transform $\widetilde{C}$ is in  $\mathcal{S}'$ but not in  $\mathcal{S}$, then $\widetilde{C}$ is in the linear system   $|-K_{\widetilde{\mathbb{F}_{e,r}}}|$. 	
		%The only curves in $\mathbb{F}_{e,r}$ whose strict transforms are in  $\mathcal{S}'$ but not in  $\mathcal{S}$ are curves in the linear system $|-K_{\widetilde{\mathbb{F}_{e,r}}}|$. 
		We will now show that for any such curve, the Seshadri ratio is greater than $a$. This proves the proposition. 
		
		Suppose that $C$ is a curve such that $\widetilde{C}$ is in the linear system $|-K_{\widetilde{\mathbb{F}_{e,r}}}|$. Then $$C \sim 2H_e + (e+2)F_e - \sum\limits_{i=1}^r E_i.$$ In particular, $\text{mult}_x C \leq 2$. Then 
		\begin{eqnarray}\label{computing L.-K}
			\frac{L\cdot C}{\text{mult}_x C} & \geq& \frac{L\cdot C}{2} \nonumber \\
			& = & \frac{-2ae + a(e+2) + 2b - \sum\limits_{i=1}^r m_i}{2} \nonumber \\
			& = & \frac{-ae + 2(a+b) - \sum\limits_{i=1}^r m_i}{2} \nonumber \\
			& > & \frac{-ae + 2a + 2ae - \sum\limits_{i=1}^r m_i}{2}~~~~~ (\text{here}~ b > ae~\text{as } L\text{ is ample}) \nonumber \\
			& = & a + \frac{ae - \sum\limits_{i=1}^r m_i}{2}.
		\end{eqnarray}
		
		We claim that $a + \frac{ae - \sum\limits_{i=1}^r m_i}{2} > a$. To see this, observe that since $L$ is ample, we have $L\cdot \widetilde{f_i} > 0$, where  $\widetilde{f_i}$ is the strict transform on $\Fer$ of the fiber $f_i$ in $\Fe$ containing $p_i$. 
		So $a-m_i >0$, and so, $a > m_i$ for each $i$. Now, as $r \leq e-1<e$, we have $\sum\limits_{i=1}^r m_i < ar < ae$. This means $ae - \sum\limits_{i=1}^r m_i > 0$. This proves the claim. Using this in the inequality $\eqref{computing L.-K}$, we see that $\frac{L\cdot C}{\text{mult}_x C} > a$. 
	\end{proof}
	Proposition \ref{our way to compute Seshadri constant} says that to compute the Seshadri constant of $L$ at $x \in \mathbb{F}_{e,r}$, it is important to know the elements of $\mathcal{S}$. We prove some results in this direction. 
	%Here, we assume that $x \notin E_i$ for any $i \in \{1,2,\dots, r\}$ and $x \notin \widetilde{C}_e$, the strict transform of $C_e$ in $\mathbb{F}_{e,r}$. Also, 
	Here, by abuse of notation, we write $H_e, F_e~\text{and}~E_i$ for the pull backs of the corresponding classes in $\widetilde{\mathbb{F}_{e,r}}$. Similarly, if $C$ is a curve in $\Fe$, we use $\widetilde{C}$ to denote its strict transform in $\Fer$ as well as in $\widetilde{\Fer}$.
	\begin{lem}\label{a = 0}
		Let $e >  0$ and $r \le e-1$.  Let $p_1,\dots, p_r \in \mathbb{F}_e$ and $x \in \mathbb{F}_{e,r}$. Let $C$ be a reduced and irreducible curve in $\Fer$ such that its strict transform $\widetilde{C} = a H_e + b F_e - \sum\limits_{i=1}^r m_i E_i - m_x E_x$ is a $(-1)$ or $(-2)$-curve in $\widetilde{\mathbb{F}_{e,r}}$. 
		%, where $a,b, m_i, m_x \in \mathbb{Z}$. 
		Then one of the following is true.
		\begin{itemize}
			\item[(i)] $a = 0$ and $b=1$,\\
			\item[(ii)] $a =1$ and $b=0$,  \\
			\item[(iii)] $a=0$ and $b=0$.
		\end{itemize}
		
	\end{lem}
	\begin{proof} {First, observe that from \cite[Corollary 2.18, Chapter V]{Hartshorne1977}, for any  reduced irreducible curve $aC_e+bf$ in $\Fe$,  we have $a\ge0,b\ge0$. Now since $\widetilde{C}$ is either the strict transform of some reduced irreducible curve in $\Fe$ or $\widetilde{C}$ is an exceptional divisor, we conclude that $a\ge0,b\ge0$ for $\widetilde{C}$ as well. }
		
		Suppose that $\widetilde{C}$ is a $(-2)$-curve. Then since $\widetilde{C}^2 = -2$ and $K_{\widetilde{\mathbb{F}_{e,r}}}\cdot \widetilde{C} = 0$, we have
		\begin{eqnarray*}
			a^2 e - 2ab +\sum\limits_{i=1}^r m_i^2 +m_x^2 -2 & = & 0~~~ \text{and} \\
			-2ae + a(e+2) + 2b - \sum\limits_{i=1}^r m_i -m_x & = & 0.
		\end{eqnarray*}
		
		Hence 
		%\begin{eqnarray} 
		$$	-2ae + a(e+2) + 2b - \sum\limits_{i=1}^r m_i -m_x  =  a^2 e - 2ab + 
		\sum\limits_{i=1}^r m_i^2 + m_x^2 - 2,$$ which in turn gives 
		\begin{eqnarray} \label{for -2 curve}
			-ae + 2a + 2b - a^2 e + 2ab + 2  =  \sum\limits_{i=1}^r m_i^2 + m_x^2 + \sum\limits_{i=1}^r m_i + m_x.
		\end{eqnarray}
		% Now since $C = aH_e + b F_e - \sum\limits_{i=1}^r m_i E_i$ in $\mathbb{F}_{e,r}$ is also reduced and irreducible, 
		
		Clearly, we have one of the four possibilities, namely, $a \neq 0$ and $b \neq 0$, or $a \neq 0$ and $b =0$, or  $a =0$ and $b \neq 0$, or  $a=0$ and $b=0$. \\
		
		We claim that the first of these possibilities does not occur. For, if $b \neq 0$, then since $C$ will be a strict transform of the curve $\mathcal{C} = aC_e + b f$ in $\mathbb{F}_e$ and $\mathcal{C}$ is irreducible, we have $b \geq ae$ (\cite[Corollary 2.18(b), Chapter V]{Hartshorne1977}). Using this on the left hand side of \eqref{for -2 curve}, we get
		\begin{eqnarray*}
			-ae + 2a + 2b - a^2 e + 2ab + 2 & \geq & ae + 2a + a^2 e + 2. 
		\end{eqnarray*}
		Now suppose that $a \neq 0$. Then since each $m_i \leq a$ and $m_x \leq a$, we have 
		\begin{equation}\label{-2 curve}
			\begin{split}
				ae + 2a + a^2 e + 2  & >  a^2 e + ae  \\
				& \geq  a^2 (r+1) + a (r+1)  \\
				& \geq  \sum\limits_{i=1}^r m_i^2 + m_x^2 + \sum\limits_{i=1}^r m_i + m_x.
			\end{split}
		\end{equation}
		This is a contradiction to  (3.3). So $a$ must be $0$ if $b \neq 0$. As $a = 0$, we must have $b=1$. 
		
		Now suppose that $b = 0$. Again by \cite[Corollary 2.18(b), Chapter V]{Hartshorne1977}, $a=0~\text{or}~ 1$. This completes the proof for the case of $(-2)$-curves. 
		
		{Now suppose that $\widetilde{C}$ is a $(-1)$-curve. Then since $\widetilde{C}^2 = -1$ and $K_{\widetilde{\mathbb{F}_{e,r}}}\cdot \widetilde{C} = -1$, we have
			\begin{eqnarray*}
				a^2 e - 2ab +\sum\limits_{i=1}^r m_i^2 +m_x^2 -1 & = & 0~~~ \text{and} \\
				-2ae + a(e+2) + 2b - \sum\limits_{i=1}^r m_i -m_x - 1& = & 0.
			\end{eqnarray*}
			Hence 
			%\begin{eqnarray} 
			$$	-2ae + a(e+2) + 2b - \sum\limits_{i=1}^r m_i -m_x -1 =  a^2 e - 2ab + 
			\sum\limits_{i=1}^r m_i^2 + m_x^2 - 1,$$ which in turn gives 
			\begin{eqnarray} \label{for -1 curve}
				-ae + 2a + 2b - a^2 e + 2ab   =  \sum\limits_{i=1}^r m_i^2 + m_x^2 + \sum\limits_{i=1}^r m_i + m_x.
			\end{eqnarray}
			Note that $ae+2a+a^2e>a^2e+ae$ when $a\neq 0$ and $b\neq 0$, so with similar steps as in \eqref{-2 curve} we can get a contradiction to \eqref{for -1 curve}. So we conclude that only the cases listed in the statement of the lemma can occur  in the $(-1)$-curve case as well. 
		}
	\end{proof} 
	We now proceed to list  all the $(-1)$ and $(-2)$-curves passing through a very general point $x \in \mathbb{F}_{e,r}$, whenever $r = e$ or $e+1$ and $p_1,\dots, p_r$ in $\mathbb{F}_e$  are also very general. First, the following lemma.
	\begin{lem}\label{b at least sum of M_i's}
		Let $e \ge 0$ and $r \leq e+1$. Let $p_1,\dots, p_r$ be distinct  points in $\mathbb{F}_e$. Suppose that 
		{$\mathcal{C} \in | aC_e + bf|$} is a reduced and irreducible curve in $\mathbb{F}_e$ having multiplicity $m_i $ at $p_i$ for each $i$.
		% where $a,b,m_i \in \mathbb{Z}$. 
		If $a>1$ or $b > e$, then $b \geq \sum\limits_{i=1}^r m_i$.
	\end{lem}
	\begin{proof}
		Let us consider the divisor $\mathcal{C}_1 = C_e + ef$. By {Riemann}-Roch theorem, we have $$h^0(\mathcal{C}_1) > \frac{\mathcal{C}_1^2-K_{\mathbb{F}_e}\cdot \mathcal{C}_1}{2}.$$ 
		
		As $\mathcal{C}_1^2 = e$ and $-K_{\mathbb{F}_e}\cdot \mathcal{C}_1 = (2C_e + (e+2)f)\cdot (C_e+ef) = e+2$, we have $h^0(\mathcal{C}_1) > e+1$. Now since $r \leq e+1$, there exists a curve $D \in |\mathcal{C}_1|$ passing through points $p_1,p_2,\dots,p_r$. Let $n_i = \text{mult}_{p_i}D$, for $1 \leq i \leq r$. Then $n_i \geq 1$ for each $i$.  {Note that $D-\mathcal{C}\in |(1-a)C_e+(e-b)f|$. By the assumptions on $a$ and $b$ and \cite[Corollary 2.18, Chapter V]{Hartshorne1977}, $D-\mathcal{C}$ cannot be effective. Hence} $\mathcal{C}$ is not a component of $D$. Therefore, we have
		\begin{eqnarray} \label{b bigger than or equal to sum of m_i's}
			b & = & (aC_e+bf)\cdot (C_e+ef) \nonumber \\ & = & \mathcal{C}\cdot \mathcal{C}_1 \nonumber \\ & = & \mathcal{C}\cdot D  \nonumber \\ 
			& \geq & \sum\limits_{i=1}^r m_in_i \nonumber \\
			& \geq & \sum\limits_{i=1}^r m_i. \nonumber
		\end{eqnarray}
		This proves the lemma.
	\end{proof}
	\begin{lem}\label{-1 and -2 curves for r=e and e+1}
		Let $e \ge 0$ and $r = e$ or $r=e+1$. Let $p_1,\dots,p_r$ be very general points in $\mathbb{F}_e$ and let $x$ be a very general point in $\mathbb{F}_{e,r}$. Let $C$ be a reduced and irreducible curve in $\Fer$ such that its strict transform $\widetilde{C} = a H_e + b F_e - \sum\limits_{i=1}^r m_i E_i - m_x E_x$ is a $(-1)$ or $(-2)$-curve in $\widetilde{\mathbb{F}_{e,r}}$. 
		%where $a,b, m_i, m_x \in \mathbb{Z}$. 
		Then one of the following is true.
		
		\begin{itemize}
			\item[(i)] $a = 0$ and $b=1$, \\
			\item[(ii)] $a =1$ and $b=0$,  \\
			\item[(iii)] $a=0$ and $b=0$,\\
			\item[(iv)] $a=1$ and $b=e$.
		\end{itemize}
	\end{lem}
	\begin{proof}
		Suppose that $\widetilde{C}$ is a $(-2)$-curve in $\widetilde{\mathbb{F}_{e,r}}$. We proceed exactly as in Lemma \ref{a = 0} to obtain the equation  
		\begin{equation}\label{for -2 curves}
			-ae + 2a + 2b - a^2 e + 2ab + 2  =  \sum\limits_{i=1}^r m_i^2 + m_x^2 + \sum\limits_{i=1}^r m_i + m_x.
		\end{equation}
		Assume $a>1$ or $b\notin \{0,1,e\}$. Then since $a \geq m_i$ for each $i$ and $a \geq m_x$, we have $2a \geq m_i + m_x$ for each $i$.
		
		Further, let $\mathcal{C} \in |aC_e + bf|$ be the curve in $\Fe$ corresponding to $\widetilde{C}$. {Since  $x \in \mathbb{F}_{e,r}$ and $p_1,\dots,p_r \in \mathbb{F}_e$ are very general points, by Lemma \ref{Xumulti},} we have $$2ab - a^2e = \mathcal{C}^2 \geq \sum\limits_{i=1}^r m_i^2 + m_x^2 -m,$$		
		where $m_x=\text{mult}_{\pi(x)}\mathcal{C}$ and $m = \min\limits_{1\le i \le r+1}  \{m_i \vert m_i \neq 0\}$ (here we set $m_{r+1} = m_x$).  Considering all these arguments, we have 
		\begin{eqnarray*}
			-ae + 2a + 2b - a^2 e + 2ab + 2  &=& 2a + (b-ae) + b + (2ab - a^2e) +2 \nonumber \\
			& \overset{\text{(*)}}\geq & (m+m_x) + (b-ae) + \sum\limits_{i=1}^r m_i + 2 + \bigg(\sum\limits_{i=1}^r m_i^2 + m_x^2 -m \bigg) \nonumber \\ %~~~~(\text{by Lemma \ref{b at least sum of M_i's}, we have}~b \geq \sum\limits_{i=1}^r m_i) \nonumber \\
			& > & m_x + \sum\limits_{i=1}^r m_i + \sum\limits_{i=1}^r m_i^2 + m_x^2,
		\end{eqnarray*}
		which is a contradiction to  \eqref{for -2 curves}. For the inequality $(*)$, note that 
		by Lemma \ref{b at least sum of M_i's}, we have $b \geq \sum\limits_{i=1}^r m_i$. {Hence we have $a\le1$ and $b\in\{0,1,e\}$. Now if $a=0$, then by \cite[Corollary 2.18, Chapter V]{Hartshorne1977}, $b$ is forced to be $0$ or $1$. And if $a=1$, then again by \cite[Corollary 2.18, Chapter V]{Hartshorne1977}, $b=0$ or $b\ge e$.} This proves the lemma for the case of a $(-2)$-curve.
		
		Suppose that $\widetilde{C}$ is a $(-1)$-curve. Then since $\widetilde{C}^2 = -1$ and $K_{\widetilde{\mathbb{F}_{e,r}}}\cdot \widetilde{C} = -1$, we have
		\begin{eqnarray*}
			-a^2 e + 2ab - \sum\limits_{i=1}^r m_i^2 - m_x^2 + 1 & = & 0 ~~~ \text{and} \\
			2ae - a(e+2) - 2b + \sum\limits_{i=1}^r m_i + m_x  + 1 & = & 0.
		\end{eqnarray*}
		This implies
		\begin{eqnarray}\label{for -1 curve}
			-ae + 2a + 2b -a^2 e + 2ab  & = & \sum\limits_{i=1}^r m_i^2 + m_x^2 + \sum\limits_{i=1}^r m_i + m_x.
		\end{eqnarray}
		We now prove that if $b > e$, then $a$ has to be $0$, by considering the following cases. \\
		
		\textbf{Case 1 :} Suppose that $b > e ~\text{and}~b-ae > 0$. If $a\neq 0$, we have 
		\begin{eqnarray*}
			-ae + 2a + 2b -a^2e + 2ab & =  & 2a + (b-ae) + b + (2ab- a^2e) \nonumber \\
			& > & (m + m_x) +  \sum\limits_{i=1}^r m_i  \nonumber \\ 
			& & + \bigg (\sum\limits_{i=1}^r m_i^2 + m_x^2 - m \bigg )~~(\text{as}~b-ae > 0) \nonumber \\
			& = & m_x + \sum\limits_{i=1}^r m_i + \sum\limits_{i=1}^r m_i^2 + m_x^2,
		\end{eqnarray*}
		which is a contradiction to \eqref{for -1 curve}. So, if $b > e~\text{and}~b-ae > 0$, we are forced to have $a=0$. \\
		
		\textbf{Case 2 :} Suppose that $b > e,~ b-ae = 0~\text{and}~r=e+1$. Then, by Lemma \ref{b at least sum of M_i's}, we have $b \geq \sum\limits_{i=1}^r m_i$. So, if $a\neq 0$, we have $a > m_j$ for some $j$ (for, otherwise, $ae = b \geq a(e+1)$, a contradiction). Therefore, $2a > m+m_j$, and 
		\begin{eqnarray*}
			-ae + 2a + 2b -a^2e + 2ab & =  & 2a + (b-ae) + b + (2ab- a^2e) \nonumber \\
			& > & (m + m_j) +  \bigg(\sum\limits_{i=1}^r m_i + m_x - m_j \bigg) \nonumber \\ 
			& & + \bigg (\sum\limits_{i=1}^r m_i^2 + m_x^2 - m \bigg )~~(\text{as}~b \geq \sum\limits_{i=1}^r m_i + m_x  - m_j) \nonumber \\
			& = & m_x + \sum\limits_{i=1}^r m_i + \sum\limits_{i=1}^r m_i^2 + m_x^2,
		\end{eqnarray*}
		again, a contradiction to  \eqref{for -1 curve}. Note that $b \geq \sum\limits_{i=1}^r m_i + m_x  - m_j$ follows from Lemma \ref{b at least sum of M_i's}. \\
		
		\textbf{Case 3 :} 
		Suppose that $b > e,~ b-ae = 0~\text{and}~r=e$. Let $m_{r+1} = m_x$. Now, since $r+1 = e+1$, applying Lemma \ref{b at least sum of M_i's} to the points $p_1,\dots,p_r,\pi(x)$, we have $b \geq \sum\limits_{i=1}^{r+1} m_i$. Therefore, if $a\neq 0$, there exists a $j \in \{1,\dots, r+1\}$ such that $a > m_j$. Then we have $2a > m + m_j$.
		% where $m = \min (\{m_i \vert m_i \neq 0\},m_x)$. 
		This implies 
		\begin{eqnarray*}
			-ae + 2a + 2b -a^2e + 2ab & =  & 2a + (b-ae) + b + (2ab- a^2e) \nonumber \\
			& > & (m + m_j) +  \sum\limits_{i=1}^r m_i + m_x - m_j \nonumber \\ 
			& & + \bigg (\sum\limits_{i=1}^r m_i^2 + m_x^2 - m \bigg )~~(\text{as}~b \geq \sum\limits_{i=1}^r m_i + m_x - m_j) \nonumber \\
			& = & m_x + \sum\limits_{i=1}^r m_i + \sum\limits_{i=1}^r m_i^2 + m_x^2,
		\end{eqnarray*}
		again a contradiction to  \eqref{for -1 curve}.\\
		
		From these three cases, we can conclude that when $a \neq 0$, $b \leq e$. However, the irreducibility criterion {(\cite[Corollary 2.18, Chapter V]{Hartshorne1977})} implies that when $a \neq 0$, either $b = 0$ or $b \geq ae$. Combining these arguments, we have that when $a \neq 0$, it should be equal to $1$, in which case $b = 0$ or $b = e$. Finally, also, by the irreducibility criterion, when $a = 0$, $b$ has to be $0$ or $1$. This completes the proof. 
	\end{proof}
	
	\begin{lem}\label{fixed components of anti-canonical}
		Let $e \ge 0$ and $r \le e+1$. Let $p_1,\dots,p_r \in \mathbb{F}_e$  and $x \in \mathbb{F}_{e,r}$. Then the following are true. \\
		\begin{itemize}
			\item[(i)] The set of all fixed components of $|-K_{\widetilde{\mathbb{F}_{e,r}}}|$ is contained in  the 
			set $$\{\widetilde{C}_e, \widetilde{f_1},\dots,\widetilde{f_r},E_1,\dots,E_r,E_x\},$$ where $f_i$ are the fibers in $\mathbb{F}_e$ containing $p_i$, $\widetilde{f_i}$ and $\widetilde{C}_e$ are the strict transforms in $\widetilde{\mathbb{F}_{e,r}}$ of  $f_i$ and $C_e$, respectively.\\
			\item[(ii)] If $e \geq 3$, $\widetilde{C}_e$ is a fixed component of $|-K_{\widetilde{\mathbb{F}_{e,r}}}|$.\\
			\item[(iii)] If $p_{i_1},\dots, p_{i_k}$ lie on the same fiber, where $k > 2$, then $\widetilde{f_{i_1}}$ is a fixed component of $|-K_{\widetilde{\mathbb{F}_{e,r}}}|$. Further, if $e \geq 3$, it is enough to have $k > 1$. 
		\end{itemize}
	\end{lem}
	\begin{proof}
		(i) Assume first that each $p_i$ belongs to a distinct fiber $f_i$ in $\mathbb{F}_e$ and that $\pi(x) \in \mathbb{F}_e$ is not in any of these fibers and also not on $C_e$. Let $f_x \subset \mathbb{F}_{e,r}$ denote the strict transform of the fiber in $\Fe$ containing $\pi(x)$. Suppose that $q \in \widetilde{\mathbb{F}_{e,r}}$ is an arbitrary point such that $q$ is not contained in any of the curves  
		$\widetilde{C}_e, \widetilde{f_1},\dots,\widetilde{f_r}, \widetilde{f_x}, E_1,\dots,E_r,E_x$. We claim that $q$ is not in base locus of $|-K_{\widetilde{\mathbb{F}_{e,r}}}|$. 
		
		To prove this claim, we first choose $e+1-r$ fibers in $\mathbb{F}_e$ which do not contain any element of the set 
		$\{ p_1,\dots, p_r, \pi(x), \text{ the image of }q \text{ in } \mathbb{F}_e\}$. 
		%the image of $q$ in $\mathbb{F}_e$ (under the natural map of blow-ups)  is not in any of these chosen fibers, and also, such that these fibers are different from the ones containing the points $p_i$ and $\pi(x)$. 
		If we denote these fibers as $f'_1,\dots, f'_{e+1-r}$ and their strict transforms as $\widetilde{f'_1}, \dots, \widetilde{f}'_{e+1-r}$, then $$2\widetilde{C}_e + \sum\limits_{i=1}^r \widetilde{f_i} + \widetilde{f_x} + \sum\limits_{i=1}^{e+1-r}\widetilde{f'_i}$$ is a curve in $\widetilde{\mathbb{F}_{e,r}}$ not containing $q$. So, $q$ is not in the base locus of the divisor class $$2H_e + \sum\limits_{i=1}^r (F_e - E_i) + (F_e-E_x) + (e+1-r)F_e,$$ which means $q$ is not in the base locus of the linear system $|-K_{\widetilde{\mathbb{F}_{e,r}}}|$.

		Suppose that $\pi(x) \in \mathbb{F}_e$ lies on $C_e$. In this case, proceeding as before, $$2\widetilde{C}_e + \sum\limits_{i=1}^r \widetilde{f_i} + \sum\limits_{i=1}^{e+2-r}\widetilde{f'_i}+E_x,$$ is a curve in $\widetilde{\mathbb{F}_{e,r}}$ not containing $q$. From this, we can conclude that $q$ is not a base point of the divisor class 
		\[2H_e + \sum\limits_{i=1}^r (F_e - E_i) + (e+2-r)F_e-E_x \in |-K_{\widetilde{\mathbb{F}_{e,r}}}|.\]
		
		By the choice of $q$, we conclude that 
		$q$ is not a base point of $|-K_{\widetilde{\mathbb{F}_{e,r}}}|$. A similar argument can be applied if the image of $x$ is in the fiber $f_i$ containing the point $p_i$ or if distinct $p_i$'s belong to the same fiber. This means that no point $q$ outside $\widetilde{C}_e, \widetilde{f_1},\dots,\widetilde{f_r},E_1,\dots,E_r,E_x$, can be a base point of $|-K_{\widetilde{\mathbb{F}_{e,r}}}|$. Therefore, the fixed components of $|-K_{\widetilde{\mathbb{F}_{e,r}}}|$ are contained in $\{\widetilde{C}_e, \widetilde{f_1},\dots,\widetilde{f_r},E_1,\dots,E_r,E_x\}$. \\
		
		(ii) Note that $-K_{\widetilde{\mathbb{F}_{e,r}}}\cdot \widetilde{C}_{e} \le -e + 2$, and  {$-e+2<0 \Leftrightarrow e \geq 3$}. This means, for $e \geq 3$, $\widetilde{C}_{e}$ is a fixed component of $|-K_{\widetilde{\mathbb{F}_{e,r}}}|$.  \\
		
		(iii) If $p_{i_1},\dots, p_{i_k}$  lie on the same fiber $f_{i_1}$ (where $k>1$) then $\widetilde{f}_{i_1} = \pi^{*}(f_{i_1}) - E_{i_1}-\dots -E_{i_k}$. So, $-K_{\widetilde{\mathbb{F}_{e,r}}}\cdot \widetilde{f}_{i_1} = -k + 2$. This is negative if and only if $k > 2$. So, when $k > 2$, we conclude that  $\widetilde{f}_{i_1}$ is a fixed component of $|-K_{\widetilde{\mathbb{F}_{e,r}}}|$. 
		
		Now suppose that $e \geq 3$. Then we already have seen that $\widetilde{C}_e$ is a fixed component of $|-K_{\widetilde{\mathbb{F}_{e,r}}}|$. So, $(-K_{\widetilde{\mathbb{F}_{e,r}}}-\widetilde{C}_{e}).\widetilde{f}_{i_1} \le -k + 1$. { So, if $k>1$ we have $(-K_{\widetilde{\mathbb{F}_{e,r}}}-\widetilde{C}_{e}).\widetilde{f}_{i_1}<0$. } Therefore when $e \geq 3$, $\widetilde{f_{i_1}}$ is a fixed component of $|-K_{\widetilde{\mathbb{F}_{e,r}}}|$, provided $k > 1$ and $p_{i_1},\dots, p_{i_k}$ are in the same fiber $f_{i_1}$.   
	\end{proof}
	\begin{theorem}\label{actual computation of Seshadri constants}
		Let $e > 0$ and $r \leq e-1$. Let $p_1,\dots,p_r \in \mathbb{F}_e$ be distinct points such that for each $i$, $p_i \notin C_e$, and for $i,j \in \{1,\dots, r\}$ with $i \neq j$,  $p_i$ and $p_j$ are not in the same fiber. Let $L = aH_e+bF_e-\sum\limits_{i=1}^r m_iE_i$ be an ample line bundle on $\mathbb{F}_{e,r}$ and $x \in \mathbb{F}_{e,r}$. 
		% where $a,b,m_i \in \mathbb{Z}$. 
		Then we have the following.
		\begin{itemize}
			\item[(i)] If $x$ is not contained in any of the curves $\widetilde{C}_e,\widetilde{f_1},\dots,\widetilde{f_r},E_1,\dots,E_r$, then $\varepsilon(\mathbb{F}_{e,r},L;x) = a$.\\
			\item[(ii)] If $x$ is not contained in any of the curves $\widetilde{C}_e,E_1,\dots,E_r$ and $x \in \widetilde{f_{i}}$ for some $i \in \{1,\dots, r\}$, then $\varepsilon(\mathbb{F}_{e,r},L;x) = a-m_i$. \\
			\item[(iii)] If $x \in \widetilde{C}_e \cap \widetilde{f_{i}}$ for some $i$, then $\varepsilon(\mathbb{F}_{e,r},L;x) = \min (b-ae,a-m_i)$.\\
			\item[(iv)] If $x \in \widetilde{f_{i}} \cap E_i$ for some $i$, then $\varepsilon(\mathbb{F}_{e,r},L;x) = \min (m_i,a-m_i)$.\\
			\item[(v)] If $x$ is not contained in any of the curves $\widetilde{f_1},\dots,\widetilde{f_r},E_1,\dots,E_r$ and $x \in \widetilde{C}_e$, then $\varepsilon(\mathbb{F}_{e,r},L;x) = \min (a,b-ae)$.\\
			\item[(vi)] If $x$ is not contained in any of the curves $\widetilde{C}_e,\widetilde{f_1},\dots,\widetilde{f_r}$ and $x \in E_i$ for some $i$, then $\varepsilon(\mathbb{F}_{e,r},L;x) = m_i$.
		\end{itemize}
	\end{theorem}
	\begin{proof}
		By  \eqref{epsilon1}, we have
		\begin{equation*}
			\varepsilon(\mathbb{F}_{e,r}, L; x) = \inf \Bigg\{\frac{L\cdot C}{\text{mult}_x C}~ \Bigg \vert ~ \widetilde{C} \in \mathcal{S}\Bigg\},
		\end{equation*}
		where $\mathcal{S}$ is the set consisting of all $(-1)$-curves, $(-2)$-curves, and the fixed components of $|-K_{\widetilde{\mathbb{F}_{e,r}}}|$  and $C$ is a curve in $\mathbb{F}_{e,r}$ with $\widetilde{C}$ as its strict transform in $\widetilde{\mathbb{F}_{e,r}}$. \\
		
		(i) Suppose that $x$ is not contained in $\widetilde{C}_e,\widetilde{f_1},\widetilde{f_2},\dots,\widetilde{f_r},E_1,\dots,E_r$. Let $f_x$ denote the strict transform on $\Fer$ of the fiber in $\mathbb{F}_{e}$ that contains $\pi(x)$. Then since it is a smooth curve containing $x$, we have $\text{mult}_x f_x = 1$. So, it is clear that $\widetilde{f_x}$ {$\subset$} $\widetilde{\mathbb{F}_{e,r}}$ is a $(-1)$-curve, and hence, an element of $\mathcal{S}$. Therefore, 
		\begin{equation*}
			\frac{L\cdot f_x}{\text{mult}_x f_x}= L\cdot f_x = a.
		\end{equation*}
		Now, applying Lemmas \ref{a = 0} and \ref{fixed components of anti-canonical}, we can conclude that { if $C \subset \mathbb{F}_{e,r}$ is any other curve then  $x \notin C$ or $\widetilde{C} \notin \mathcal{S}$}. This implies $f_x$ is the Seshadri curve and $\varepsilon(\mathbb{F}_{e,r}, L, x) = a$. This completes (i). \\
		
		(ii) Suppose that $x$ is not contained in $\widetilde{C}_e,E_1,\dots,E_r$ and $x \in \widetilde{f_{i}}$ for some $i \in \{1,2,\dots, r\}$. By the same arguments as in (i), $\widetilde{f_i}$ will be the Seshadri curve for that particular $i$ for which $x \in \widetilde{f_i}$. In this case, 
		\begin{equation*}
			\frac{L\cdot \widetilde{f_i}}{\text{mult}_x \widetilde{f_i}}= L\cdot \widetilde{f_i} = a-m_i,
		\end{equation*}
		thereby completing (ii). \\
		
		(iii) Suppose that $x \in \widetilde{C}_e \cap \widetilde{f_{i}}$ for some $i$. Then by the same arguments as above, we only need to compute $L\cdot \widetilde{C}_e$ and compare it with $a-m_i$. We have $L\cdot \widetilde{C}_e = b - ae$. This implies that either $\widetilde{C}_e$ or $\widetilde{f_i}$ is the Seshadri curve and $\min (b-ae,a-m_i)$ is the Seshadri constant. \\
		
		(iv) Suppose that $x \in \widetilde{f_{i}} \cap E_i$ for some $i$. We have $L\cdot E_i = m_i$. Therefore the Seshadri constant is $\min (m_i,a-m_i)$. \\
		
		(v) Suppose that $x$ is not contained in $\widetilde{f_1},\dots,\widetilde{f_r},E_1,\dots,E_r$ and $x \in \widetilde{C}_e$. We have $L\cdot \widetilde{C}_e = b-ae$. So, in this case, either $\widetilde{C}_e$ or $f_x$ will be the Seshadri curve, and so, the Seshadri constant is $\min (a,b-ae)$.\\
		
		(vi) Suppose that $x$ is not contained in $\widetilde{C}_e,\widetilde{f_1},\dots,\widetilde{f_r}$ and $x \in E_i$. Then since $L\cdot E_i = m_i < a$, we  conclude that $E_i$ is the Seshadri curve and $m_i$ is the Seshadri constant. 
	\end{proof}
	\begin{remark}
		\begin{itemize}
			\item[(i)] {Note that we assumed that the points $p_1,\dots,p_r$ are such that for each $i$, $p_i \notin C_e$, and for $i,j \in \{1,\dots, r\}$ with $i \neq j$,  $p_i$ and $p_j$ are not in the same fiber. This implies that $x \notin \widetilde{C}_e \cap E_i$ for any $i$. For, 
				if $x \in \widetilde{C}_e \cap E_i$ then 
				$p_i = \pi(x)  \in C_e$, thereby contradicting the hypothesis.  Similarly $x \notin \widetilde{f_i} \cap E_j$, where $i \neq j$ (again, if this is false then  $p_i$ and $p_j$ are in the same fiber in $\mathbb{F}_e$).}
			\item[(ii)]  As a consequence of the above,  Theorem \ref{actual computation of Seshadri constants} 
			computes the Seshadri constant of $L$ for all $x \in \Fer$. Moreover, note that the six cases considered in Theorem \ref{actual computation of Seshadri constants} are mutually exclusive.		\end{itemize}
	\end{remark}
	\begin{theorem}\label{for x in general position}
		Let $e > 0$ and $r = e$ or $r=e+1$. Let $p_1,\dots,p_r \in \mathbb{F}_e$ be very general points. Let $L = aH_e+bF_e-\sum\limits_{i=1}^r m_iE_i$ be an ample line bundle on $\mathbb{F}_{e,r}$ and let $x \in \mathbb{F}_{e,r}$ be a very general point. %, where $a,b,m_i \in \mathbb{Z}$.
		Then $$\varepsilon(\Fer,L;x)=\min (a,b-\sum m_i),$$ where the sum runs over  the largest $e$  integers among 
		$\{m_1,\dots,m_r\}$.
		\begin{proof}
			{By Lemma \ref{fixed components of anti-canonical}, there are only finitely many fixed components of $|-K_{\widetilde{\mathbb{F}_{e,r}}}|$. Hence there are only finitely many curves $C\subset \Fer$ whose strict transform is a fixed component of $|-K_{\widetilde{\mathbb{F}_{e,r}}}|$. As union of these finitely many curves is a proper closed subset, we can choose $x$ from outside this union as $x$ is a very general point.} So, by Proposition \ref{our way to compute Seshadri constant}, to compute Seshadri constant at $x$, it is enough to take the infimum of Seshadri ratios with respect to curves $C$ passing through $x$ with $\widetilde{C}$ a $(-1)$ or $(-2)$-curve in $\widetilde{\Fer}$. Let $C = \alpha H_e + \beta F_e -n_1E_1-\cdots-n_rE_r$ be such a curve. \\%We consider both the $(-2)$ case and $(-1)$ case separately. \\

			\textbf{Case 1 :} 
			Suppose that $\widetilde{C}$ is a $(-2)$-curve.
			By Lemma \ref{-1 and -2 curves for r=e and e+1}, we only have the following four possibilities. \\
			
			(i) $\alpha=0,\beta=0$.
			In this case, $C$ is an exceptional divisor. As $x$ is a very general point, $C$ cannot pass through $x$. \\
			
			(ii)  $\alpha=1,\beta=0$,
			i.e., $C=\widetilde{C}_e$. Again, as $x$ is very general, $C$ cannot pass through $x$.\\
			
			(iii) $\alpha=0,\beta=1$. In this case,
			as $\widetilde{C}^2=-2$, $C$ has to be $F_e-E_i$ for some $i$. But this cannot happen since $x$ is very general. \\
			
			(iv) $\alpha=1,\beta=e$.
			So $C=H_e+eF_e-n_1E_1-\cdots-n_rE_r$. {As $C\cdot(F_e-E_i)=1-n_i$, and $C$ and $F_e-E_i$ are distinct reduced irreducible curves, we have $1-n_i\geq 0$}. So,
			$$\widetilde{C}^2=e-\bigg(\sum_{i=1}^{r}n_i^2+n_x^2\bigg)=-2,$$
			where $n_x=\text{mult}_xC=1$. Therefore $\sum\limits_{i=1}^r n_i^2=e+1$. This is possible only if $r=e+1$ and $n_i=1$ for all $i$, i.e., the image of $C$ in $\mathbb{F}_e$ (under the usual map of blow-up) passes through $e+2$ very general points $p_1,\dots,p_r,\pi(x)$. But we have $h^0(C_e+ef)=e+2$. This means, no curve in the linear system $|C_e + ef|$ can pass through $e+2$ very general points, which is a contradiction. \\

			\textbf{Case 2 :} 		  Suppose that $\widetilde{C}$ is a $(-1)$-curve. We again consider the same four possibilities.\\
			
			(i) If $\alpha=0,\beta=0$, no such curve exists.\\
			% $\alpha=1,\beta=0$, the same argument as in \textbf{Case 1}  will work. \\
			
			(ii) If $\alpha=1,\beta=0$, then $C = \widetilde{C}_e$. However, as $e > 0$ and $x$ is very general, $\widetilde{C}_e$ cannot pass through $x$. \\
			
			(iii) If $\alpha=0,\beta=1$, then since $\widetilde{C}^2=-1$, $C$ is the strict transform of the fiber in $\Fe$ passing through $\pi(x)$. So the Seshadri ratio with respect to $C$ is $a$. \\
			
			(iv) If $\alpha=1,\beta=e$, then with a similar explanation as in \textbf{Case 1}, we get $\sum\limits_{i=1}^{r}n_i^2=e$. This implies that $e$ of $\{n_1,\dots, n_r\}$ are 1 and others are zero. That is, any curve $C \in |C_e + ef|$ whose strict transform is a $(-1)$-curve and $\pi(x) \in C$ must pass through exactly $e$ points (with multiplicity $1$) of $p_1,\dots, p_r$. For any $e$ points of the set $\{p_1,\dots, p_r\}$, the existence of such a curve $C$ is guaranteed as $h^0(C_e + ef) = e+2$. So, for computing Seshadri constant, it is enough to get the smallest among Seshadri ratios with respect to such curves. This is precisely $b-\sum m_i$, where the sum runs over the largest $e$ numbers among $\{m_1,\dots,m_r\}$.\\
			
			From the above arguments, we can conclude that $$\varepsilon(\Fer,L;x)=\min (a,b-\sum m_i),$$ where the sum runs over the largest $e$ numbers among $\{m_1,\dots,m_r\}.$
		\end{proof}
	\end{theorem}
	
	\begin{remark}\label{p2blowup}
		The hypothesis in Theorem \ref{actual computation of Seshadri constants} is not applicable to the case $e=0$, since for any point $p\in \mathbb{F}_0$, there is a $C_0$ passing through $p$. Note that 
		$\mathbb{F}_{0,0} = \mathbb{F}_0 = \mathbb{P}^1 \times \mathbb{P}^1$ and Seshadri constants on 
		$\mathbb{P}^1 \times \mathbb{P}^1$ are well-known. For 
		$r\geq 1$, $\mathbb{F}_{0,r}$ is isomorphic to $\mathbb{P}^2$ blown-up at $r+1$ general points. As mentioned earlier, 
		general blow-ups of $\mathbb{P}^2$ are well-studied. 
		
		Similarly $\mathbb{F}_{1,r}$ is isomorphic to $\mathbb{P}^2$ blown-up at $r+1$ general points, for $r\geq 0$.
	\end{remark}
	
	\begin{remark}
		By Theorem \ref{actual computation of Seshadri constants}, for $r\leq e-1$, the Seshadri constant of $L=aH_e+bF_e-m_1E_1-\cdots -m_rE_r$ at a very general point $x$ in $\Fer$ is $a$. But when $r=e$ or $r=e+1$, Theorem \ref{for x in general position} says that the Seshadri constant at $x$ can be different from $a$. 
		For an explicit example, consider the line bundle $L=3H_e+4F_e-2E_1$ on $\mathbb{F}_{1,1}$. By Proposition \ref{ample}, $L$ is ample. By Theorem \ref{for x in general position}, 
		$\varepsilon(\mathbb{F}_{1,1},L;x)=2 < 3 =a $.
		
		%		To exhibit the existence of such a situation, we now give an example of such an ample line bundle. 
		%By the following proposition, line bundle $L=3C_e+4F_e-2E_1$ is ample in $\mathbb{F}_{1,1}$ and has 
		%$\varepsilon(\mathbb{F}_{1,1},L;x)=2\neq a$. 	
		
		The following result characterizes the ampleness of line bundles on $\mathbb{F}_{e,r}$.		
		
	\end{remark}
	
	\begin{proposition}\label{ample}
		Let $e > 0$ and $r \leq e+1$.  Let $p_1,\dots,p_r \in \mathbb{F}_e$ be distinct points such that for each $i$, $p_i \notin C_e$, and for $i,j \in \{1,\dots, r\}$ with $i \neq j$, $p_i$ and $p_j$ are not on the same fiber. Let $L = aH_e+bF_e-\sum\limits_{i=1}^r m_iE_i$ be a line bundle on $\mathbb{F}_{e,r}$. Then $L$ is ample if and only if the following conditions hold. 
		
		\begin{itemize}
			\item[(1)] $a>m_i>0$~~~$\forall~ 1\leq i\leq r$,
			\item[(2)] $b>ae$, and
			\item[(3)] $b>m_1+\cdots+m_r$.
		\end{itemize}
		\begin{proof}
			Suppose that $L$ is an ample line bundle. Then $L\cdot (F_e-E_i)=a-m_i>0$, $L\cdot E_i=m_i>0$, $L\cdot H_e=b-ae>0$. We know that $h^0(C_e+ef)>e+1 \geq r$. It guarantees the existence of a curve $C$ passing through $p_1,\dots,p_r$. Therefore $L\cdot \widetilde{C}=b-\sum\limits_{i=1}^r m_i>0$.
			
			Conversely, we assume the three conditions mentioned in the proposition and show that $L$ is ample using the Nakai-Moishezon criterion. We have
			$$L^2= -a^2e+2ab-\sum\limits_{i=1}^r m_i^2=a(b-ae)+ab-\sum\limits_{i=1}^r m_i^2.$$ This is positive because $a>0,b-ae>0$ and $ab>\sum\limits_{i=1}^r am_i>\sum\limits_{i=1}^r m_i^2$.	It now remains to prove that $L\cdot C>0$ for all irreducible curves $C$.\\
			
			Let $C = \alpha H_e+\beta F_e -n_1E_1-\cdots -n_rE_r$. As before, we will do it case by case. \\
			
			(i) Suppose that $ \alpha =0, \beta =0$. In this case, $C=E_i$ for some $i$. So, $L\cdot C=L\cdot E_i=m_i>0$. \\
			
			(ii) Suppose that $\alpha=1, \beta=0$. In this case, $C\sim H_e$. Therefore, $L\cdot C=b-ae>0$. \\
			
			(iii) Suppose that $ \alpha =0, \beta=1$. Then $C=F_e$ or $C=F_e-E_i$ for some $i$. In both the cases, $L \cdot C >0$. In fact, the exact values are, respectively, $a$ and $a-m_i$. \\
			
			(iv) Suppose that $\alpha=1, \beta =e$. Then $C=H_e+eF_e-n_1E_1-\cdots-n_rE_r$. As $\alpha=1$, we have $n_i\leq 1$. Therefore,
			$ L \cdot C=b-\sum\limits_{i=1}^r m_in_i~\geq~ b-\sum\limits_{i=1}^r m_i>0$. \\
			
			(v) Finally, suppose that $\alpha >1$ and $\beta >e$. Then $$L \cdot C=-a\alpha e+\alpha b + a\beta -\sum\limits_{i=1}^rm_in_i=\alpha (b-ae)+a\beta-\sum\limits_{i=1}^rm_in_i.$$ Now, by Lemma \ref{b at least sum of M_i's}, we have $\beta \geq \sum\limits_{i=1}^r n_i$.  So, $a\beta \geq \sum\limits_{i=1}^r a n_i \geq \sum\limits_{i=1}^r m_in_i$. Therefore, we conclude that $L\cdot  C >0$ in this case also.
			
			Hence $L$ is ample.
		\end{proof}
	\end{proposition}
	\begin{question} \label{qns}
		The above results lead naturally to the following questions. 
		\begin{enumerate}
			\item Is $r = e+1$ the optimal value for Theorem \ref{for x in general position} to hold? In other words, does the theorem hold when $r\ge e+2$? 
			\item Can the Seshadri constant of an ample line bundle on $\Fer$ be a non-integer, for some  $r$? 
		\end{enumerate}
		
		Note that Seshadri constants of ample line bundles on $\Fer$ are always integers when $r \le e+1$, by 
		Theorems \ref{actual computation of Seshadri constants} and \ref{for x in general position} (when $r = e$ or $r = e+1$, we consider Seshadri constants at very general points). So to answer the Question \ref{qns}(2), we need to consider $r \ge e+2$. 
		
		We answer the first question in Example \ref{exm3.13} below. Well-known results about Seshadri constants on the blow-ups of $\mathbb{P}^2$ give a positive answer for the second question when $e \le 1$ (see Remark \ref{p2blowup} and \cite[Remark 4.7(2)]{Garcia2008}).  We answer the second question for $e > 1$ in Example 
		\ref{exm3.14} below. 
	\end{question}
	
	\begin{example}\label{exm3.13}
		Consider the surface $\mathbb{F}_{1,3}$ obtained by blowing up $\mathbb{F}_1$ at three very general points. Let $L=3H_1+5F_1-2E_1-2E_2-2E_3$ be a line bundle on $\mathbb{F}_{1,3}$.
		
		First, we will show that $L$ is an ample line bundle. Clearly $L^2>0$, $L\cdot H_1>0$, $L \cdot F_1 > 0$, $L\cdot E_i>0$, $L\cdot (F_1-E_i)>0$, $L\cdot (H_1 + F_1)>0$ and $L\cdot (H_1 +F_1 - E_i)>0$ for $i=1,2,3$. 
		
		It remains to check for irreducible curves $C =\alpha H_1+\beta F_1 -n_1E_1-n_2E_2-n_3E_3$ with $\alpha>1$ and $\beta >1$. By Lemma \ref{b at least sum of M_i's}, $\beta \geq n_i+n_j$, where $i\neq j$ and $i,j \in \{1,2,3\}$. So we  conclude $3\beta \geq 2n_1+2n_2+2n_3$. Now $L\cdot C=2\alpha+3\beta - 2n_1 - 2n_2 - 2n_3>0,$ since $\alpha>0$. So $L$ is an ample line bundle. As $h^0(C_1+2f)>4$ (by Riemann-Roch), for a very general point $x \in \mathbb{F}_{1,3}$, there is a curve $C$ in $|H_1+2F_1-E_1-E_2-E_3|$ passing through $x$. Therefore $\varepsilon(\mathbb{F}_{1,3},L;x)\leq \frac{L\cdot C}{1}=2$. This example confirms that $r=e+1$ is indeed optimal.
	\end{example}
	
	\begin{example}\label{exm3.14}
		For the second question, consider the surface $\mathbb{F}_{3,6}$ obtained by blowing up $\mathbb{F}_3$ at six very general points and the line bundle $L =6H_3+19F_3-4E_1-\cdots-4E_6$ on $\mathbb{F}_{3,6}$. First, we will show that $L$ is an ample line bundle. Clearly $L^2=24>0$, $L\cdot E_i=4>0, L\cdot F_3=6>0, L\cdot (F_3-E_i)=2>0, L.H_3=1>0$. Now let $C'$ be the curve $H_3+3F_3-n_1E_1-\cdots-n_6E_6$, where $n_i\leq 1$. Then we have $L\cdot C'>0$. So it remains to show $L\cdot C>0$ for all irreducible curves $C =\alpha H_3+\beta F_3-n_1E_1-\cdots-n_6E_6\subset \mathbb{F}_{3,6}$ with $\alpha>1$ and $\beta>3$. By Lemma \ref{b at least sum of M_i's}, $\beta \geq n_{i_1}+n_{i_2}+n_{i_3}+n_{i_4}$ for any distinct $i_1,i_2,i_3,i_4 \in \{1,2,\cdots, 6\}$. So, we can conclude that $6\beta\geq 4(n_1+\cdots+n_6)$. This gives $L\cdot C=\alpha+6\beta - 4(n_1+\cdots+n_6)>0$. So $L$ is ample.
		
		Now to compute the Seshadri constant of $L$ at a very general point $x\in \mathbb{F}_{3,6}$, we can choose $x$ to be outside the curves whose strict transforms on $\widetilde{\mathbb{F}_{e,r}} = \text{Bl}_x(\mathbb{F}_{e,r})$ are the fixed components of $|-K_{\widetilde{\mathbb{F}_{e,r}}}|$.
		So, by Proposition \ref{our way to compute Seshadri constant}, it is enough to compute the infimum of Seshadri ratios with respect to curves passing through $x$ whose strict transforms on $\widetilde{\mathbb{F}_{e,r}}$ are $(-1)$ or $(-2)$-curves.\\
		Let
		$$S=\Big\{D=\alpha H_3+\beta F_3-n_1E_1-\cdots-n_6E_6-n_xE_x \Big | D^2=K_{\widetilde{\mathbb{F}_{e,r}}}\cdot D=-1 \text{ or } D^2=K_{\widetilde{\mathbb{F}_{e,r}}}\cdot D - 2 =-2\Big\}.$$ By \cite[Table 1]{Lee-Shin}, there are 480 elements in $S$. We know that the Seshadri constant is achieved by an irreducible curve passing through $x$ whose strict transform is a (-1) or a (-2)-curve on $\widetilde{\mathbb{F}_{e,r}}$. This gives us $\alpha\geq 0,\beta \geq 3\alpha, n_x\geq 1$. In fact, going through 480 elements in $S$ by a computer calculation, we see that there are only 77 elements satisfying these conditions. Denoting $E_1+E_2+\cdots+E_6$ by $E$, we list out all these $77$ possibilities below. 
		\begin{enumerate}
			\item $E_i-E_x$ ~~ $\forall ~ 1\leq i\leq 6$.
			\item $F_3-E_x$.
			\item $F_3-E_i-E_x$ ~~ $\forall~ 1\leq i\leq 6$.
			\item $H_3+3F_3-E_i-E_j-E_k-E_x$ ~~ $\forall~ 1\leq i<j<k\leq 6$.
			\item $H_3+3F_3-E+E_i+E_j-E_x$ ~~ $\forall~ 1\leq i<j\leq 6$.
			\item $H_3+4F_3-E+E_i-E_x$ ~~ $\forall~ 1\leq i\leq 6$.
			\item $H_3+4F_3-E-E_x$.
			\item $2H_3+6F_3-E+E_i+E_j-E_x$ ~~ $\forall~ 1\leq i<j\leq 6$.
			\item $2H_3+6F_3-E+E_i-2E_x$ ~~ $\forall~ 1\leq i\leq 6$.
			\item $3H_3+9F_3-2E-2E_x$.
		\end{enumerate}
		We can see that as $x$ is very general, the divisors (1) and (3) are not effective. If divisors in (5) are effective, there are curves in the linear system $|C_3+3f|$ that pass through 5 very general points, which is not possible as $h^0(C_3+3f)=5$. So they are not effective. Similarly, we can see that the divisor in (7) is also not effective.
		
		For curves $C\subset \Fer$ whose strict transforms are the remaining divisors in the above list, the least value of $\frac{L\cdot C}{n_x}$ is attained by curves whose strict transforms are divisors in (10). By the Reimann-Roch Theorem, we have $h^0(3C_3+9f)>21$. So, we can choose $C\in |3H_3+9F_3-2E_1-\cdots-2E_6|$ passing through $x$ with multiplicity at least $2$. We know that the multiplicity has to be less than or equal to 3. In fact, if $\text{mult}_x = 3$, the Seshadri ratio will be $\frac{9}{3} = 3$. So $\varepsilon(\mathbb{F}_{3,6},L;x)\le 3$.
		
		But one can check that there is no $(-1)$ or $(-2)$-curve for which the Seshadri ratio is less than or equal to 3 by going through the above list. So  $\text{mult}_xC=2$ and 
		$\frac{L\cdot C}{\text{mult}_xC}=\frac{9}{2}.$
		Hence we conclude 
		$$\varepsilon(\mathbb{F}_{3,6},L;x)=4.5.$$
	\end{example}

	\section{Linear systems on $\Fe$}\label{LinearSystem} 
	Let $a,b$ be non-negative integers. 
	Let $\Li(a,b)$ denote the complete linear system $|aC_e+bf|$ on $\Fe$. Let $p_1,\dots,p_r$ be very general points in $\Fe$. For non-negative integers $m_1,\dots,m_r$, let $\Li(a,b,m_1,\dots,m_r)$ denote the linear system of curves in $\Li(a,b)$ passing through $p_1,\dots,p_r$ with multiplicities at least $m_1,\dots,m_r$ respectively.
	
	Define the \textit{virtual dimension} and the\textit{ expected dimension} of $\Lia$, denoted by $v (\Lia)$ and $e(\Lia)$ respectively, as follows:
	$$v(\Lia) = \text{dim}(\Li(a,b))-\sum_{i=1}^{r}{m_i+1 \choose 2},$$
	$$e(\Lia) = \text{max}  (v(\Lia),-1).$$\
	
	Clearly, $\text{dim}( \Lia) \geq e(\Lia)$. If this inequality is strict, we call $\Lia$ \textit{special}, and otherwise, \textit{non-special}.
	
	Let $\Fer$ denote the blow-up of $\Fe$ at $p_1,\dots,p_r$ and $\Li$  denote the linear system of curves consisting strict transform of curves in $\Lia$. We define the virtual dimension and expected dimension of $\Li$ to be the same as the corresponding values of $\Lia$: $$v(\Li)=v(\Lia) \text{ and }e(\Li)=e(\Lia).$$ 
	
	By \cite[Proposition 2.3]{Laf2002},  $v(\Li)=\frac{\Li^2-K_{\mathbb{F}_{e,r}}.\Li}{2}$. 
	
	We now recall the notion of a \textit{$(-1)$-special} linear system defined in \cite{Laf2002} by means of the following algorithm.
	
	\begin{algorithm}\label{algo}
		Let $\Li$ be a linear system as above.  
		%Let C denote a general curve in $\Li$.
		\begin{itemize}
			\item[Step 1 :] If $E$ is a $(-1)$-curve with $-t =\Li\cdot E<0$, replace $\Li$ by $\Li - tE$. 
			\item[Step 2 :] If $\Li \cdot  \widetilde{C}_e < 0$, replace $\Li$ by $\Li -\widetilde{C}_e$, where $\widetilde{C}_e$ is the strict transform of $C_e$.
			\item[Step 3 :] If $ \Li \cdot  \widetilde{C}_e \geq 0$ and $\Li\cdot E\geq 0$ for every $(-1)$-curve $E$, stop the process. Else, go to Step 1.
		\end{itemize}
	\end{algorithm}
	Note that this process will end after finitely many steps. Let $\mathcal{M}$ denote the linear system obtained after the above procedure is complete. Then $\text{dim}(\mathcal{M})=\text{dim}(\Li)$.
	\begin{definition}
		A linear system $\Li$ is called $(-1)$-special if $v(\mathcal{M}) > v(\Li)$.
	\end{definition}
	
	Suppose that $\Li$ is $(-1)$-special. Then we have $$\text{dim}(\Li) = \text{dim} (\mathcal{M}) \geq v(\mathcal{M})>v(\Li).$$ So  a $(-1)$-special linear system is always special. It is natural to ask if the converse holds. In this direction, the following conjecture was proposed for a Hirzebruch surface by Laface (\cite[ Conjecture 2.6]{Laf2002}) as an analogue of the Hirschowitz version of the SHGH conjecture.
	\begin{conjecture}\label{conj1}
		A linear system $\Lia$ on $\mathbb{F}_e$ is special if and only if its strict transform is $(-1)$-special.
	\end{conjecture}
	In this direction we propose the following.
	\begin{conjecture}\label{conj2}
		If general curve of a non-empty linear system $\Li$ on $\mathbb{F}_{e,r}$ is reduced and $\widetilde{C}_e$ is not a fixed component of $\Li$, then $\Li$ is non-special.
	\end{conjecture}

	\begin{lemma} \label{-1 special Lemma}
		If a linear system $\Li$ is $(-1)$-special, then $\Li$ is non-reduced or $\widetilde{C}_e$ is a fixed component of $\Li$.
	\end{lemma}
	\begin{proof}
		If $\Li$ is $(-1)$-special and non-reduced, there is nothing to prove. So we assume $\Li$ is reduced and prove that in this case $\widetilde{C}_e$ is a fixed component of $\Li$. As $\Li$ is reduced, $\Li \cdot E \geq -1$ for every $(-1)$-curve $E$. This is because $\Li \cdot E < -1$ implies $(\Li - E) \cdot E < 0$. This means $2E$ is a fixed component of the linear system $\Li$, thereby contradicting the fact that $\Li$ is reduced. Now, suppose that $\Li \cdot E \ge 0$ for all $(-1)$-curves $E$. Then the only way for virtual dimension to increase during Algorithm \ref{algo} is when $\Li \cdot \widetilde{C}_e < 0$. This implies $\widetilde{C}_e$ is a component of $|\Li|$, thereby completing the proof.
		
		Now, if $E$ is a $(-1)$-curve such that $\Li \cdot E = -1$, we have
		$$v(\Li-E) = \frac{(\Li - E)^2 - K_{\Fer} \cdot (\Li - E)}{2} = \frac{\Li^2 - 2\Li \cdot E + E^2 - K_{\Fer} \cdot \Li + K_{\Fer} \cdot E}{2} = v(\Li).$$ So in this case, we proceed to Step 2 of Algorithm \ref{algo}. If $(\Li - E) \cdot \widetilde{C}_e < 0$, then $\widetilde{C}_e$ is a component of $|\Li - E|$, and so, a component of $|\Li|$ as well (as $E$ is a fixed component of $\Li$). If in Step 2 $(\Li - E) \cdot \widetilde{C}_e \ge 0$, we go to 
		Step 3. But since we know that $\Li$ is $(-1)$-special, the virtual dimension has to increase. So if $D$ is another $(-1)$-curve such that $(\Li-E)\cdot D=-1$, we can see that $v(\Li-E-D)=v(\Li)$. Again if $(\Li-E-D)\cdot \widetilde{C}_e<0$, $\widetilde{C}_e$ is a fixed component of $|L|$ as before. If not, we again proceed to Step 1. But then, using similar arguments as above, we see that the virtual dimension does not increase in Step 1. This eventually implies that $\widetilde{C}_e$ is a fixed component of $\Li$.
		%	Step 1 as virtual dimension increases. Even if there is some other $(-1)$-curve $D$ with $(\Li-E)\cdot D=-1$ we can see that $v(\Li-E-D)=v(\Li)$. So in order to increase the virtual dimension we need to have $(\Li-E-D)\cdot \widetilde{C}_e<0$, in this case also with the above argument we can see that $\widetilde{C}_e$ is fixed component of $\Li$. In fact if $L$ is reduced, to increase the virtual dimension $\widetilde{C}_e$ should be a fixed component of $\Li$. 
	\end{proof}
	\begin{theorem}
		Conjecture \ref{conj1} implies Conjecture \ref{conj2}.   
	\end{theorem}
	
	\begin{proof}
		
		We prove the contrapositive to Conjecture \ref{conj2}. Let $\mathcal{L}$ be a special linear system. Conjecture \ref{conj1} implies that $\mathcal{L}$ is $(-1)$-special. By Lemma \ref{-1 special Lemma}, this implies $\Li$ is non-reduced or $\widetilde{C}_e$ is a fixed component of $\Li$.
	\end{proof}
	\begin{question}
		Is the Conjecture \ref{conj1} equivalent to Conjecture \ref{conj2}?
	\end{question}
	
	We now propose the following conjecture for a Hirzebruch surface analogous to the $(-1)$-curves conjecture for $\mathbb{P}^2$.
	\begin{conjecture}\label{conj3}
		Let $\Fer$ denote the blow-up of  $\Fe$ at $r$ very general points $p_1,\dots,,p_r$. If $C$ is an irreducible and reduced curve in $\Fer$ with negative self intersection, then $C$ is either a $(-1)$-curve or the strict transform of $C_e$ on $\Fer$.
	\end{conjecture}
	\begin{proposition}
		Conjecture \ref{conj2} implies Conjecture \ref{conj3}. 
		\begin{proof}
			Let $C$ denote an irreducible and reduced curve in $\Fer$ with $C^2<0$. 
			Suppose that $C\neq \widetilde{C}_e$. Let $|C|$ denote the linear system of curves linearly equivalent to $C$. Then since $C^2<0$, $|C|=\{C\}$. Hence $|C|$ is a reduced linear system. Further, $\widetilde{C}_e$ is not a component of $C$. Therefore, by Conjecture \ref{conj2},  $v(|C|)=\text{dim}|C|=0$. So $C^2=K_{\mathbb{F}_{e,r}}\cdot C$ and the adjunction formula gives $C^2=K_{\mathbb{F}_{e,r}}\cdot C=-1$.
		\end{proof}
	\end{proposition}
	We now prove Conjecture \ref{conj3} when $r \le e+2$.
	\begin{theorem}
		Let $\Fer$ denote the blow-up of $\Fe$ at $r$ very general points $p_1,\dots, p_r$. For $r\leq e+2$, if $C$ is an irreducible and reduced curve in $\Fer$ with negative self intersection, then $C$ is either a $(-1)$-curve or the strict transform of $C_e$ on $\Fer$.
	\end{theorem}
	\begin{proof}
		Let $C$ denote an irreducible and reduced curve in $\Fer$ with $C^2<0$. 	Suppose that $C=aH_e+bF_e-n_1E_1-\dots-n_rE_r$ for some non-negative integers $a, b, n_1,\ldots, n_r$. Then 
		\begin{eqnarray*}
			C^2 & = &-a^2e+2ab-\sum_{i=1}^{r}n_i^2~~~\text{and} \\
			K_{\mathbb{F}_{e,r}}\cdot C & = &2ae-a(e+2)-2b+\sum_{i=1}^{r}n_i.
		\end{eqnarray*}
		Let $C^2=-k$ for some $k\geq 1$ and $K_{\mathbb{F}_{e,r}}\cdot C=\alpha$.
		So, we have
		\begin{equation*}
			-a^2e+2ab-\sum_{i=1}^{r}n_i^2+k=2ae-a(e+2)-2b+\sum_{i=1}^{r}n_i-\alpha,
		\end{equation*}
		which, on rearranging, gives
		\begin{equation}\label{eqn1}
			(2ab-a^2e)+2a+b+(b-ae)+k+\alpha=\sum_{i=1}^{r}n_i^2+\sum_{i=1}^{r}n_i.
		\end{equation}
		
		If $a=0$, by the irreducibility of $C$, we have $b=0$ or $1$. If $b=1$, $C=F-E_{i_1}-\dots-E_{i_k}$ for some $\{i_1,\dots, i_k\}\subseteq \{1,\dots, r\}$. Since $p_1,\dots,p_r$ are very general, $k=1$, again implying that $C$ is a $(-1)$-curve. Similarly if $b=0$, $C$ is equal to  $E_i$ for some $i$.  
		
		Suppose that $a \neq 0$ and $b\notin \{0,e\}$. Since $C$ is irreducible, $b-ae\geq 0$. Also, by hypothesis, $r-1\leq e+1$. If $n_i=0$ for all $1 \le i \le r$, then $C^2=2ab-a^2e\geq 0$ which is not possible. So  $n_i\neq 0$ for some $i$. Let $n_j=\min\{n_i|n_i\neq 0, 1\leq i\leq r\}$. By Lemma \ref{b at least sum of M_i's}, $b\geq \sum\limits_{i=1}^{r}n_i-n_j$. Further, by {Lemma \ref{Xumulti},} $2ab-a^2e \geq \sum\limits_{i=1}^{r}n_i^2 - n_j$. Now since $a\neq 0$, we have $2a\geq 2n_j$. Summing all these inequalities, we have
		\begin{equation}\label{4.2}
			(2ab-a^2e)+2a+b+(b-ae)\geq \sum_{i=1}^{r}n_i^2+\sum_{i=1}^{r}n_i.
		\end{equation}
		
		Using \eqref{eqn1} {and \eqref{4.2}} , we have $k+\alpha \leq 0$. Finally, by	adjunction formula, $-k+\alpha   \geq -2$. If $k\geq 2$, then $\alpha \geq k - 2 \geq 0$. So $k+\alpha \geq 2$, which contradicts $k+\alpha \leq 0$. Thus $k=1$ and hence $\alpha=-1$. Therefore, $C$ is a $(-1)$-curve when $ a \neq 0 \text{ and } b \notin \{0,e\}$.
		
		If $a\neq 0$ and $b=0$, as $C$ is irreducible, we have $a=1$. So, $C=\widetilde{C}_e$. If $a\neq 0$ and $b=e$, the irreducibility of $C$ will again imply that $a=1$.  Let $C=H_e+eF_e-n_1E_1-\cdots -n_rE_r$. As $a=1$, $n_i\leq 1$ for all $1\leq i\leq r$.  Also since $C^2=e-\sum\limits_{i=1}^{r}n_i^2<0$, $\sum\limits_{i=1}^{r}n_i^2>e$. However, if $\sum\limits_{i=1}^{r}n_i^2>e+1$, image of $C$ in $\Fe$ will be a curve passing through at least $e+2$ very general points. But that is a contradiction since $h^0(C_e+ef)=e+2$. So $\sum\limits_{i=1}^{r}n_i^2=e+1$. Hence $C^2=-1$ and $-K_{\mathbb{F}_{e,r}}\cdot C=e+2-\sum\limits_{i=1}^r n_i=1$, since $\sum\limits_{i=1}^{r}n_i^2=e+1$ and $n_i\leq 1$. Therefore, $C$ is a $(-1)$ curve in this case as well, thereby, completing the proof.
	\end{proof}
	
	Dumnicki, K\"{u}ronya, Maclean and Szemberg \cite[Main Theorem]{Dumnicki2016} showed that the SHGH conjecture for $\mathbb{P}^2$ implies the  irrationality of  Seshadri constant of a suitable ample line bundle at a very general point of blow-up of $\mathbb{P}^2$ at $r$ very general points. Hanumanthu and Harbourne \cite[Theorem 2.4]{Hanumanthu2018} improved this by showing that the Weak SHGH Conjecture for $\mathbb{P}^2$ gives the same result. It is natural to ask the following question. 
	
	\begin{question}\label{irratiional}
		Do Conjectures \ref{conj1}, \ref{conj2} or \ref{conj3} imply the irrationality of Seshadri constant of some ample line bundle at a very general point in the blow-up of $\Fe$ at $r$ very general points? 
	\end{question}
	
	By  Theorem \ref{actual computation of Seshadri constants} and Theorem \ref{for x in general position},  for $r\leq e+1$, Seshadri constant for any ample line bundle at a very general point is an integer. So an affirmative answer for Question \ref{irratiional} is possible only for $r \geq e+2$.
	
	\section*{Acknowledgements}
	We thank the referee for a careful reading of the paper and numerous suggestions which improved the paper. We thank Ronnie Sebastian for several useful comments.
	Authors were partially supported by a grant from Infosys Foundation. The fourth author was supported by the National Board for Higher Mathematics (NBHM), Department of Atomic Energy, Government of India (0204/2/2022/R\&D-II/2683).

\end{document}